\documentclass{amsart}
 \usepackage{amssymb, latexsym}
 \usepackage[dvips]{graphics}

\DeclareMathOperator{\gc}{\mathrm{GCD}}
\DeclareMathOperator{\mo}{\mathrm{mod}}
\DeclareMathOperator{\ar}{\mathrm{Area}}
\DeclareMathOperator{\len}{\mathrm{Length}}
\DeclareMathOperator{\mx}{\mathrm{max}}

\begin{document}
 \title{The Distribution of Special Subsets of the Farey Sequence}
 \author{Alan Haynes}

 \newtheorem{theorem}{Theorem}
 \newtheorem{lemma}{Lemma}
 \newtheorem{corollary}{Corollary}
 \newcommand{\mc}{\mathcal}
\newcommand{\fqp}{\mc{F}_{Q,p}}
\newcommand{\fq}{\mc{F}_Q}
\newcommand{\zpz}{\mathbb{Z}/p\mathbb{Z}}
\newcommand{\mbb}{\mathbb}

\begin{abstract}
We will examine the subset $\mathcal{F}_{Q,p}$ of Farey fractions of order $Q$ consisting of those fractions whose denominators are not divisible by a fixed prime $p$. In particular, we will provide an asymptotic result on the distribution of $H-$tuples of consecutive fractions in $\mathcal{F}_{Q,p}$, as $Q\rightarrow\infty$.
\end{abstract}
\maketitle

\section{Introduction}

For $Q\in\mathbb{N}^*$, the Farey fractions of order $Q$ are defined as  \[\mathcal{F}_Q=\left\{\frac{a}{q}\in\mathbb{Q} : 0\leq\frac{a}{q}\leq 1, 1\leq q\leq Q,\,\mathop{\mathrm{GCD}}(a,q)=1\right\},\]
where $\mathop{\mathrm{GCD}}$ denotes the greatest common divisor function. Farey fractions have been studied in the past mainly for two reasons. First, they are important in the field of diophantine approximation. Second, there is a connection between Farey fractions and the Riemann Hypothesis ([4],[7]). There is significant motivation to study the distribution of subsets of $\mathcal{F}_Q$ satisfying congruence conditions on the numerators and denominators. For example, the analog for Dirichlet L-functions of the results of Franel and Landau ([4],[7]) involve such subsets. As a first attempt to study these kinds of distributions, we might begin by considering a subset of $\mathcal{F}_Q$ defined as \[\mathcal{F}_{Q,\text{odd}}=\left\{\frac{a}{q}\in\mathcal{F}_Q:q\text{ odd}\right\}.\] It is well known that if $a_1/q_1<a_2/q_2$ are consecutive elements of $\mathcal{F}_Q$, then $q_1a_2-a_1q_2=1$. However, this result does not hold for consecutive elements of $\mathcal{F}_{Q,\text{odd}}.$ In view of this, we could define for $k\in\mathbb{N}^*$ the numbers \[N_{Q,\text{odd}}(k)=\#\left\{\frac{a_1}{q_1}<\frac{a_2}{q_2}\text{ consecutive in } \mathcal{F}_{Q,\text{odd}} : q_1a_2-a_1q_2=k\right\}.\] It has been proved [6] that the asymptotic frequencies $\rho _{\text{odd}}(k)$ defined by \[\rho_{\text{odd}}(k)=\lim_{Q\to\infty}\frac{N_{Q,\text{odd}}(k)}{\#\mathcal{F}_{Q,\text{odd}}}\] exist and, further, that they can be computed exactly as \[\rho _{\text{odd}}(k)=\frac{4}{k(k+1)(k+2)}.\] In the introduction of [6], the following two questions were presented as open problems.
\begin{enumerate}
  \item Choose two positive integers $k_1$ and $k_2$ and three consecutive elements $a_1/q_1<a_2/q_2<a_3/q_3$ of $\mathcal{F}_{Q,\text{odd}}$. In the limit as $Q$ goes to infinity, what is the probability that $q_1a_2-a_1q_2=k_1$ and $q_2a_3-a_2q_3=k_2$?
  \item What can be deduced about subsets of $\mathcal{F}_{Q}$ with more general congruence conditions on the numerators and denominators? Is it possible to generalize the above results to include such cases?
\end{enumerate}
An answer to the first question and to the more general problem involving $H-$tuples of consecutive elements in $\mathcal{F}_{Q,\text{odd}}$ was provided in [2]. In this paper we will begin to answer the second question by extending the results in [2] to subsets $\mathcal{F}_{Q,p}$ of $\mathcal{F}_Q$ defined as \[\mathcal{F}_{Q,p}=\left\{\frac{a}{q}\in\mathcal{F}_Q:p\nmid q\right\},\] where $p$ is allowed to be any fixed prime.

\section{Preliminary Work}
Before we can precisely state our main theorem, we will need to make some definitions. Let $\mathcal{T}=\{ (x,y)\in [0,1]\times [0,1]:x+y>1\}$ and for $n\in\mathbb{N}^*$ define the regions \[\mathcal{T}_n=\left\{ (x,y)\in \mathcal{T}:\left[\frac{1+x}{y}\right]=n\right\}.\] Also, define the map $T:\mathcal{T}\rightarrow\mathcal{T}$ by \ \[T(x,y)=\left( y,\left[\frac{1+x}{y}\right] y-x\right) ,\] and given an $R-$tuple $\{ n_1,\ldots ,n_R\}$ of positive integers, define the region \[\mathcal{T}_{n_1,\ldots n_R}=\mathcal{T}_{n_1}\cap T^{-1}\mathcal{T}_{n_2}\cap\cdots\cap T^{1-R}\mathcal{T}_{n_R}.\] Now let us explore some basic properties of the regions that we have defined. Let $a_1/q_1<a_2/q_2<a_3/q_3$ be consecutive elements of $\mc{F}_Q$. Then it follows from [5, Lemma 1] that \[ q_3=\left[\frac{Q+q_1}{q_2}\right] q_2-q_1=\left[\frac{1+q_1/Q}{q_2/Q}\right] q_2-q_1.\] In other words, $T(q_1/Q,q_2/Q)=(q_2/Q,q_3/Q).$ Also, we know from [6] that \[ q_1a_3-a_1q_3=\left[\frac{Q+q_1}{q_2}\right],\] so applying some basic properties of Farey fractions we find that for each positive integer $n,$
\begin{eqnarray}
 & &\#\left\{ \frac{a_1}{q_1}<\frac{a_2}{q_2}<\frac{a_3}
     {q_3} \text{ consecutive in } \mathcal{F}_{Q}: q_1a_3-a_1q_3=n \right\}\nonumber\\
& &=\#\left\{ \frac{a_1}{q_1}<\frac{a_2}{q_2} \text{ consecutive in } \mathcal{F}_{Q}: \left[\frac{Q+q_1}{q_2}\right]=n \right\}\\
& &=\#\left\{ (q_1,q_2)\in\mathbb{Z}^2:1\leq q_1,q_2\leq Q, q_1+q_2>Q,\right.\nonumber\\
& &\qquad\qquad\left.\gc (q_1,q_2)=1, \left[\frac{1+q_1/Q}{q_2/Q}\right]=n \right\}\nonumber\\
& &=\#\{ (x,y)\in Q\mc{T}_n\cap\mathbb{Z}^2:\gc (x,y)=1\}.\nonumber
\end{eqnarray}
As a technical detail, note that if $n=2Q$ then (1) is actually off by $1$, because in this case $(Q-1)/Q$ and $1/1$ are consecutive in $\fq$ and $\left[\frac{Q+q_1}{q_2}\right]=n,$ but there is no fraction after $1/1$ in $\fq$. Since we will eventually be interested in an asymptotic estimate as $Q\rightarrow\infty$, we can ignore this detail in our problem by assuming that $n$ is fixed and that $Q$ is large.

In general, if $R\geq 3$ and $a_1/q_1<\cdots <a_R/q_R$ are consecutive elements of $\mc{F}_Q$ then
\begin{eqnarray*}
& & T^r\left(\frac{q_1}{Q},\frac{q_2}{Q}\right)= \left(\frac{q_{1+r}}{Q},\frac{q_{2+r}}{Q}\right) \text{ and}\\
& & q_ra_{r+2}-a_rq_{r+2}=\left[\frac{Q+q_r}{q_{r+1}}\right],
\end{eqnarray*}
for each $r\in\{ 1,\ldots ,R-2\}.$  A little thought then reveals that, given an $(R-2)-$tuple $\{ n_1,\ldots ,n_{R-2}\}$ of positive integers, we have
\begin{eqnarray}
& &\#\left\{ \frac{a_1}{q_1}<\cdots <\frac{a_R}{q_R} \text{ consecutive in } \mc{F}_Q: \right.\nonumber\\
& & \qquad\qquad\left. q_ra_{r+2}-a_rq_{r+2}=n_r, r\in\{ 1,\ldots ,R-2\} \right\}\nonumber\\
& &=\#\left\{ (x,y)\in Q\mc{T}_{n_1,\ldots ,n_{R-2}}\cap\mathbb{Z}^2:\gc (x,y)=1\right\} .
\end{eqnarray}

Next, for each $H\in\mathbb{N}^*$ and for each $H-$tuple $\Delta =(\Delta_1,\ldots ,\Delta_H)\in (\mathbb{N}^*)^H$, define
\begin{eqnarray*}
N_{Q,p}(\Delta )&=&\#\left\{\frac{a_1}{q_1}<\cdots <\frac{a_{H+1}}{q_{H+1}} \text{ consecutive in } \mc{F}_{Q,p}\right.\\
& &\qquad\qquad\left. :q_ha_{h+1}-a_hq_{h+1}=\Delta_h, h\in\{ 1,\ldots ,H\} \right\}.
\end{eqnarray*}
 A basic question that we are interested in is the following. Given $\Delta =(\Delta_1,\ldots ,\Delta_H)$, which $R-$tuples of consecutive elements $a_1/q_1<\cdots <a_R/q_R$ in $\mc{F}_Q$ give rise to an $(H+1)-$tuple $a_1'/q_1'<\cdots <a_{H+1}'/q_{H+1}'$ of consecutive elements in $\fqp$ with the properties that $a_1'/q_1'= a_1/q_1,$ $ a_{H+1}'/q_{H+1}'= a_R/q_R,$ and $q_h'a_{h+1}'-a_h'q_{h+1}'=\Delta_h$ for each $h\in\{ 1,\ldots ,H\}$? In answering this question we will make use of the fact that for each $r\in\{ 1,\ldots R-1\},$ it is not possible that $p$ divides both $q_r$ and $q_{r+1}.$ This is most easily seen from the identity $q_ra_{r+1}-a_rq_{r+1}=1.$ Note here that one immediate consequence of this fact is that we may restrict our search to those $R$ for which $H+1\leq R\leq 2H+1$. Now, assuming that the elements $a_1/q_1<\cdots <a_R/q_R$ do give rise to an $(H+1)-$tuple with the desired properties, let $\{ r_1,\ldots ,r_{H+1}\}$ be the unique subset of $\{ 1,\ldots ,R\}$ for which $r_1<\cdots <r_{H+1}$ and $p\nmid q_{r_h}$ for each $h\in\{ 1,\ldots ,H+1\}$. Then we have that $a_h'/q_h'=a_{r_h}/q_{r_h}$, and we know that $r_{h+1}$ is equal to either $r_h+1$ or $r_h+2$. If $r_{h+1}=r_h+1$ then we have that \[ q_h'a_{h+1}'-a_h'q_{h+1}'=q_{r_h}a_{r_{h+1}}-a_{r_h}q_{r_{h+1}}=q_{r_h}a_{r_h+1}-a_{r_h}q_{r_h+1}=1,\] and if $r_{h+1}=r_h+2$ then we have that \[ q_h'a_{h+1}'-a_h'q_{h+1}'=q_{r_h}a_{r_{h+1}}-a_{r_h}q_{r_{h+1}}=q_{r_h}a_{r_h+2}-a_{r_h}q_{r_h+2}=\left[\frac{Q+q_{r_h}}{q_{r_{h+1}}}\right] .\] Now we may answer our question. For a given $R-$tuple of consecutive elements in $\mc{F}_Q$ to give rise to an $(H+1)-$tuple of consecutive elements in $\fqp$ as described above, the following conditions are necessary and sufficient.
\begin{enumerate}
\item Neither of $q_1,q_R$ is divisible by $p,$ and $\#\{q_r\in\{q_1,\ldots ,q_R\}:p\nmid q_r\}=H+1,$ and
\item If $\{ r_1,\ldots ,r_{H+1}\}$ is the unique subset of $\{ 1,\ldots ,R\}$ for which $r_1<\cdots <r_{H+1}$ and $p\nmid q_{r_h}$ for any $h\in\{ 1,\ldots ,H+1\},$ then the following two conditions are satisfied for each $h\in\{ 1,\ldots ,H\}$
\begin{enumerate}
\item If $r_{h+1}=r_h+1$ then $\Delta_h=1$, and
\item If $r_{h+1}=r_h+2$ then $\Delta_h=\left[\frac{Q+q_{r_h}}{q_{r_{h+1}}}\right] .$
\end{enumerate}
\end{enumerate}
Next we will develop a little more notation. An $R-$tuple $a_1/q_1<\cdots <a_R/q_R$ of elements consecutive in $\fq$ can be encoded as an element \[ (\alpha ,n)=((\alpha_1,\ldots ,\alpha_R),(n_1,\ldots ,n_{R-2}))\in (\mathbb{Z}/p\mathbb{Z})^R\times (\mathbb{N}^*)^{R-2}\] by requiring that \[ q_r\equiv\alpha_r (\mo p) \text{ for each } r\in\{ 1,\ldots ,R\}, \text{ and that}\] \[n_r=\left[\frac{Q+q_r}{q_{r+1}}\right] \text{ for each } r\in\{ 1,\ldots ,R-2\}.\] With this in mind, we might ask which elements of $(\mathbb{Z}/p\mathbb{Z})^R\times (\mathbb{N}^*)^{R-2}$ represent an $R-$tuple of consecutive elements in $\fq$. Let $(\alpha ,n)= ((\alpha_1,\ldots ,\alpha_R),(n_1,\ldots ,n_{R-2}))$ be any element of $(\mathbb{Z}/p\mathbb{Z})^R\times (\mathbb{N}^*)^{R-2}$. From our previous comments, if for any $r\in\{ 1,\ldots ,R-1\}$, both of $\alpha_r$ and $\alpha_{r+1}$ are equal to $0$ in $\zpz$, then we know that $(\alpha ,n)$ does not represent an $R-$tuple of consecutive elements in $\fq$. Also, since \[ q_{r+2}=\left[\frac{Q+q_r}{q_{r+1}}\right] q_{r+1}-q_r\] for each $r\in\{ 1,\ldots ,R-2\}$, we see that the following four additional conditions must hold if $(\alpha ,n)$ represents an $R-$tuple of consecutive elements in $\fq$.
\begin{enumerate}
\item If $\alpha_r=\alpha_{r+2}=0$ and $\alpha_{r+1}\neq 0$ then $p\mid n_r,$
\item If $\alpha_r=0, \alpha_{r+1}\neq 0,$ and $\alpha_{r+2}\neq 0$ then $n_r\equiv\alpha_{r+2}\alpha_{r+1}^{-1}(\mo p),$
\item If $\alpha_r\neq 0, \alpha_{r+1}\neq 0,$ and $\alpha_{r+2}=0$ then $n_r\equiv\alpha_r\alpha_{r+1}^{-1}(\mo p),$ and
\item If $\alpha_r\neq 0, \alpha_{r+1}\neq 0,$ and $\alpha_{r+2}\neq 0$ then $n_r\equiv (\alpha_{r+2}+\alpha_r)\alpha_{r+1}^{-1}(\mo p).$
\end{enumerate}
If $(\alpha ,n)$ satisfies all of the above conditions then comparing with (2), we find that the number of consecutive $R-$tuples in $\fq$ which are represented by $(\alpha ,n)$ is equal to \[\#\left\{ (x,y)\in Q\mc{T}_{n_1,\ldots ,n_{R-2}}\cap\mathbb{Z}^2:\gc (x,y)=1, (x,y)\equiv (\alpha_1,\alpha_2)(\mo p)\right\}.\] Since we want to isolate those $R-$tuples that give rise to $(H+1)-$tuples in $\fqp$ with the correct $\Delta$ value, we define a subset $\mc{A}_R\subseteq (\zpz )^R\times (\mathbb{N}^*)^{R-2}$ by the condition that $(\alpha ,n)\in\mc{A}_R$ if and only if
\begin{enumerate}
\item $\alpha_1$ and $\alpha_R$ are nonzero,
\item For each $r\in\{ 1,\ldots ,R-1\},$ if $\alpha_r=0$ then $\alpha_{r+1}\neq 0,$
\item The number of $r\in\{ 1,\ldots ,R\}$ for which $\alpha_r\neq 0$ is $H+1,$ and
\item For each $r\in\{ 1,\ldots ,R-2\}$, each of the above four conditions on the congruence class of $n_r(\mo p)$ holds.
\end{enumerate}
Finally, define a function $\delta:\mc{A}_R\rightarrow (\mathbb{N}^*)^H$ as follows. Given $(\alpha ,n)\in\mc{A}_R$, let $\{ r_1,\ldots ,r_{H+1}\}$ be the unique subset of $\{ 1,\ldots ,R\}$ for which $r_1<\cdots <r_{H+1}$ and $\alpha_{r_h}\neq 0$ for each $h\in\{ 1,\ldots ,H+1\}$. Then let $\delta (\alpha ,n)=(\delta_1,\ldots ,\delta_H),$ where \[\delta_h=\begin{cases} 1 \text{ if } r_{h+1}=r_h+1,\\ n_{r_h} \text{ if } r_{h+1}=r_h+2
\end{cases}.\]
From our above discussion, we have shown that
\begin{eqnarray}
N_{Q,p}(\Delta )=\sum_{R=H+1}^{2H+1}\sum_{\substack{(\alpha ,n)\in\mc{A}_R\\ \delta (\alpha ,n)=\Delta}}N_{\alpha_1,\alpha_2}^p(Q\mc{T}_{n_1,\ldots ,n_{R-2}}),
\end{eqnarray}
where
\begin{eqnarray*}
N_{\alpha_1,\alpha_2}^p(Q\mc{T}_{n_1,\ldots ,n_{R-2}})&=&\#\left\{ (x,y)\in Q\mc{T}_{n_1,\ldots ,n_{R-2}}\cap\mathbb{Z}^2:\right.\\
& &\qquad\qquad\left.\gc (x,y)=1, (x,y)\equiv (\alpha_1,\alpha_2)(\mo p) \right\}.
\end{eqnarray*}
We now state the main theorem to be proved in this paper.
\begin{theorem}
For any prime $p$, $H\in\mathbb{N}^*$, and $\Delta\in (\mathbb{N}^*)^H,$ we have that
\[\rho_p(\Delta ):=\frac{ N_{Q,p}(\Delta )}{\#\mc{F}_{Q,p}}=\sum_{R=H+1}^{2H+1}\sum_{\substack{(\alpha ,n)\in\mc{A}_R\\ \delta (\alpha ,n)=\Delta}} \frac{2\ar (\mc{T}_{n_1,\ldots ,n_{R-2}})}{p(p-1)}+O_H\left(\frac{\log^2Q}{Q}\right) .\]
\end{theorem}

The special case of $H=1$ can be formulated as a corollary.

\begin{corollary}
For any prime $p$,
\begin{align*}
\rho_p(1)&= 1-\frac{2}{3p}+O\left(\frac{\log^2Q}{Q}\right) ,\text{ and }\\ \rho_p(k)&=\frac{8}{p}\cdot\frac{1}{k(k+1)(k+2)}+ O\left(\frac{\log^2Q}{Q}\right)\quad\text{ for }k\geq 2.
\end{align*}
\end{corollary}

\section{Asymptotic Estimates}
We will prove Theorem 1 from (3) by establishing asymptotic estimates for the numbers $N_{\alpha_1,\alpha_2}^p(Q\mc{T}_{n_1,\ldots ,n_{R-2}})$ and $\#\fqp$. Let $\Omega\subseteq\mbb{R}^2$ be a convex region contained in the square $[0,Q]\times [0,Q].$ Then for any prime $p$ and any integers $a\in\{ 1,\ldots ,p-1\}$ and $ b\in\{ 0,\ldots ,p-1\}$ define the numbers
\begin{eqnarray*}
N_p(\Omega )&=&\#\{ (x,y)\in\Omega\cap\mbb{Z}^2:\gc (x,y)=1, p\nmid x\}, \text{ and}\\
N_{a,b}^p(\Omega )&=&\#\{ (x,y)\in\Omega\cap\mbb{Z}^2:\gc (x,y)=1, (x,y)\equiv (a,b)(\mo p )\}.
\end{eqnarray*}
We begin by proving the following lemma.
\begin{lemma}
If $L$ is the length of the boundary of $\Omega$ then we have that
\begin{eqnarray}
N_p(\Omega )&=&\frac{p}{p+1}\cdot\frac{6\ar (\Omega )}{\pi^2}+O(L\log Q), \text{ and}\\
N_{a,b}^p(\Omega )&=&\frac{1}{p^2-1}\cdot\frac{6\ar (\Omega )}{\pi^2}+O(L\log Q).
\end{eqnarray}
\end{lemma}
\begin{proof}
First we prove (5). Because of the fact that for any positive integer $n,$
\[\sum_{d\mid n}\mu (d)=\begin{cases}1 \text{ if } n=1\\ 0 \text{ if } n>1\end{cases},\] we can write \[ N_{a,b}^p(\Omega )=\sum_{\substack{(x,y)\in\Omega \\(x,y)\equiv (a,b)(\mo p)}}\sum_{d\mid x,y}\mu (d).\] Interchanging the order of summation and noting that $x,y\leq Q$, we have that \[ N_{a,b}^p(\Omega )=\sum_{d=1}^Q\mu (d)\cdot\#\{ (x,y)\in\Omega\cap\mbb{Z}^2:(x,y)\equiv (a,b)(\mo p),d\mid x,y\}.\] Now observe that if $p\mid d$ then \[\#\{ (x,y)\in\Omega\cap\mbb{Z}^2:(x,y)\equiv (a,b)(\mo p), d\mid x,y\}=0,\] since $a$ is not congruent to $0$ modulo $p$. On the other hand, if $p\nmid d$, then the conditions $(x,y)\equiv (a,b)(\mo p)$ and $d\mid x,y$ can be combined into the new condition $(x,y)\equiv (a',b')(\mo pd),$ for some $a',b'\in\mbb{Z}$. The number of pairs $(x,y)\in\Omega$ satisfying this new condition is \[\ar\left(\frac{\Omega}{pd}\right)+O\left(\frac{L}{pd}\right),\] so we have that
\begin{eqnarray*}
N_{a,b}^p(\Omega )&=&\sum_{\substack{1\leq d\leq Q\\(p,d)=1}}\mu (d)\left(\frac{\ar (\Omega )}{(pd)^2}+O\left(\frac{L}{pd}\right)\right)\\
&=&\frac{\ar (\Omega )}{p^2}\sum_{\substack{1\leq d\leq Q\\(p,d)=1}}\frac{\mu (d)}{d^2}+O\left(\sum_{d=1}^Q\frac{L}{d}\right)\\
&=&\frac{\ar (\Omega )}{p^2}\left(\sum_{\substack{d\geq 1\\(p,d)=1}}\frac{\mu (d)}{d^2}+O\left(\frac{1}{Q}\right)\right) +O(L\log Q)\\
&=& \frac{1}{p^2-1}\cdot\frac{6\ar (\Omega )}{\pi^2}+O(L\log Q),
\end{eqnarray*}
where we have used the fact that \[\sum_{\substack{d\geq 1\\(p,d)=1}}\frac{\mu (d)}{d^2}=\frac{\zeta (2)^{-1}}{1-p^{-2}}.\] Finally, we have that
\begin{eqnarray*}
N_p(\Omega )&=&\sum_{a=1}^{p-1}\sum_{b=0}^{p-1} N_{a,b}^p(\Omega )
=p(p-1)N_{1,0}^p(\Omega )\\
&=&\frac{p}{p+1}\cdot\frac{6\ar (\Omega )}{\pi^2}+O(L\log Q).
\end{eqnarray*}
\end{proof}
Since $\fqp$ can be defined as \[\fqp=\{(a,q)\in\mbb{Z}^2:1\leq q\leq Q, 0\leq a\leq q, \gc (a,q)=1, p\nmid q\},\] an immediate consequence of (4) is that
\begin{eqnarray}
\#\fqp =\frac{p}{p+1}\cdot\frac{3Q^2}{\pi^2}+O(Q\log Q).
\end{eqnarray}
We will also need to borrow a few miscellaneous results, which are compiled in the following lemma.
\begin{lemma} Let $R\in\{ 3,4,\ldots\}$ and $(n_1,\ldots n_{R-2})\in (\mbb{N}^*)^{R-2}.$
\begin{enumerate}
\item If $n_r\geq 4R-6$ for any $r\in\{ 1,\ldots ,R-2\}$,  then $\mc{T}_{n_1,\ldots ,n_{R-2}}=\phi,$ unless $n_1=\cdots =n_{r-2}=n_{r+2}=\cdots =n_{R-2}=2$ and $n_{r-1}=n_{r+1}=1.$
\item If $\mx (n_1,\ldots ,n_{R-2})>2Q$ then $Q\mc{T}_{n_1,\ldots ,n_{R-2}}\cap\mbb{Z}^2=\phi$.
\item We have that \[\len (\delta\mc{T}_{n_1,\ldots ,n_{R-2}})<<_R\frac{1}{n_1},\] uniformly in $n_2,\ldots n_{R-2}$ as $n_1\rightarrow\infty$.
\end{enumerate}
\end{lemma}
\begin{proof}
For the proofs of 1 and 3, see [2, Lemma 3.4]. Statement 2 follows from [3, Remark 2.3]
\end{proof}
Now we have the tools necessary to prove Theorem 1.
\begin{proof}[Proof of Theorem 1]
Applying (5) to the region $Q\mc{T}_{n_1,\ldots ,n_{R-2}}$ and using statement 3 of Lemma 2, we have that
\begin{eqnarray}
N_{a,b}^p(Q\mc{T}_{n_1,\ldots ,n_{R-2}})&=&\frac{6\ar (Q\mc{T}_{n_1,\ldots ,n_{R-2}})}{\pi^2(p^2-1)}+O_R\left(\frac{Q\log Q}{n_1}\right)\nonumber\\
&=&\frac{6Q^2\ar (\mc{T}_{n_1,\ldots ,n_{R-2}})}{\pi^2(p^2-1)}+O_R\left(\frac{Q\log Q}{n_1}\right).
\end{eqnarray}
Now, statement 1 of Lemma 2 implies that, given $n_1\in\mbb{N}^*$, the number of elements $(n_2,\ldots ,n_{R-2})\in (\mbb{N}^*)^{R-3}$ for which $\mc{T}_{n_1,\ldots ,n_{R-2}}\neq\phi$ is bounded above by a constant that depends on $R$. Recall that in our problem, $R$ is also bounded above by $2H+1.$ Combining this result with (3), with (7), and with statement 2 of Lemma 2, we have that
\begin{eqnarray*}
N_{Q,p}(\Delta )&=&\sum_{R=H+1}^{2H+1}\sum_{\substack{(\alpha ,n)\in\mc{A}_R\\ \delta (\alpha ,n)=\Delta}}N_{\alpha_1,\alpha_2}^p(Q\mc{T}_{n_1,\ldots ,n_{R-2}})\\
&=& \sum_{R=H+1}^{2H+1}\sum_{\substack{(\alpha ,n)\in\mc{A}_R\\ \delta (\alpha ,n)=\Delta}}\left(\frac{6Q^2\ar (\mc{T}_{n_1,\ldots ,n_{R-2}})}{\pi^2(p^2-1)} +O_R\left(\frac{Q\log Q}{n_1}\right)\right)\\
&=& \frac{6Q^2}{\pi^2(p^2-1)}\sum_{R=H+1}^{2H+1}\sum_{\substack{(\alpha ,n)\in\mc{A}_R\\ \delta (\alpha ,n)=\Delta}}\ar (\mc{T}_{n_1,\ldots ,n_{R-2}}) +O_H\left(\sum_{m=1}^{2Q}\frac{Q\log Q}{m}\right)\\
&=& \frac{6Q^2}{\pi^2(p^2-1)}\sum_{R=H+1}^{2H+1}\sum_{\substack{(\alpha ,n)\in\mc{A}_R\\ \delta (\alpha ,n)=\Delta}}\ar (\mc{T}_{n_1,\ldots ,n_{R-2}}) +O_H\left(Q\log^2 Q\right).
\end{eqnarray*}
Putting this result together with (6) proves Theorem 1.
\end{proof}

\end{document}